\documentclass[a4paper,11pt]{amsart}

\usepackage{cite}
\usepackage{algpseudocode}
\usepackage{amsmath, amsthm}
\usepackage{textcomp}
\usepackage{amssymb}
\usepackage[all]{xy}
\usepackage{enumerate}
\usepackage{fancyhdr}
\usepackage{lscape}
\usepackage{url}
\usepackage{setspace}
\usepackage{longtable}
\usepackage{tikz}
\usepackage{wasysym}

\usepackage[misc]{ifsym}
\usepackage{wasysym}

\theoremstyle{plain}
\newtheorem{theorem}{Theorem}[section]

\theoremstyle{definition}
\newtheorem{defi}[theorem]{Definition}
\newtheorem{lm}[theorem]{Lemma}
\newtheorem{theorem/conj}[theorem]{Theorem-Conjecture}
\newtheorem{prop}[theorem]{Proposition}

\newtheorem{conj}[theorem]{Conjecture}

\theoremstyle{remark}
\newtheorem{rmk}{Remark}

\DeclareMathOperator{\sing}{Sing}

\DeclareMathOperator{\Int}{Int}

\DeclareMathOperator{\Proj}{Proj}
\DeclareMathOperator{\Spec}{Spec}
\DeclareMathOperator{\rank}{rank}

\DeclareMathOperator{\Cox}{Cox}

\DeclareMathOperator{\Mob}{Mob}

\DeclareMathOperator{\MMob}{\overline{Mob}}
\DeclareMathOperator{\Bs}{Bs}

\DeclareMathOperator{\Gl}{GL}

\DeclareMathOperator{\Nef}{Nef}
\DeclareMathOperator{\Eff}{Eff}
\DeclareMathOperator{\NNE}{\overline{NE}}

\newcommand{\Pic}{\operatorname{Pic}}

\newcommand{\convex}{\operatorname{Convex}}

\def \Q {\mathbb{Q}}
\def \C {\mathbb{C}}
\def \Z {\mathbb{Z}}
\def \F {\mathbb{F}}

\newcommand{\x}{\mathrm{x}}
\newcommand{\y}{\mathrm{y}}
\renewcommand{\t}{\mathrm{t}}
\renewcommand{\u}{\mathrm{u}}

\newcommand{\PP}{\mathbb{P}}

\newcommand\qt{{\slash\kern-0.65ex\slash}}

\title{Singular del Pezzo fibrations and birational rigidity}

\author{Hamid Ahmadinezhad}

\keywords{Birational Automorphism; Mori Fibre Space; Sarkisov Program; Del Pezzo Fibration; Birational Rigidity.\vspace{0.2cm} }

\begin{document}

\begin{abstract} A known conjecture of Grinenko in birational geometry asserts that a Mori fibre space with the structure of del Pezzo fibration of low degree is birationally rigid if and only if its anticanonical class is an interior point in the cone of mobile divisors. The conjecture is proved to be true for smooth models (with a generality assumption for degree $3$). It is speculated that the conjecture holds for, at least, Gorenstein models in degree $1$ and $2$. In this article, I present a (Gorenstein) counterexample in degree $2$ to this conjecture.
\end{abstract}

\maketitle


\section{Introduction}
All varieties in this article are projective and defined over the field of complex numbers.
Minimal model program played on a uniruled $3$-fold results in a Mori fibre space (Mfs for short). Such output for a given variety is not necessarily unique. The structure of the endpoints is studied via the birational invariant called {\it pliability} of the Mfs, see Definition~\ref{pliability}. A Mori fibre space is a $\Q$-factorial variety with at worst terminal singularities together with a morphism $\varphi\colon X\rightarrow Z$, to a variety $Z$ of strictly smaller dimension, such that $-K_X$, the anti-canonical class of $X$, is $\varphi$-ample and $$\rank\Pic(X)-\rank\Pic(Z)=1.$$

\begin{defi}Let $X\rightarrow Z$ and $X^\prime\rightarrow Z^\prime$ be Mori fibre spaces. A birational map $f\!\colon\!X\!\dashrightarrow\!X^\prime$ is {\it square} if it fits into a commutative diagram\begin{center}$\xymatrixcolsep{3pc}\xymatrixrowsep{3pc}
\xymatrix{
X\ar@{-->}^f[r]\ar[d]& X^\prime\ar[d]\\
Z\ar@{-->}[r]^g&Z^\prime
}$\end{center} where $g$ is birational and, in addition, the map $f_L\!\colon\!X_L\!\dashrightarrow\!X^\prime_L$ induced on generic fibres  is biregular. In this case we say that $X/Z$ 
and $X^\prime/Z^\prime$ are {\it square birational}. We denote this by $X/Z\sim X^\prime/Z^\prime$.
\end{defi}

\begin{defi}[Corti~\cite{corti-mella}]\label{pliability}The {\it pliability} of a Mfs $X\rightarrow Z$ is the set
\[\mathcal{P}(X\slash Z)=\{\text{Mfs } Y\rightarrow T\mid X\text{ is birational to } Y\}\slash\sim\] 
A Mfs $X\rightarrow Z$ is said to be {\it birationally rigid} if $\mathcal{P}(X\slash Z)$ contains a single element.\end{defi}

A main goal in the birational geometry of $3$-folds is to study the geometry of Mfs, and their pliability. Note that finite pliability, and in particular birational rigidity, implies non-rationality. There are three types of Mfs in dimension $3$, depending on the dimension of $Z$. If $\dim(Z)=1$, then the fibres must be del Pezzo surfaces. When the fibration is over $\PP^1$, I denote this by $dP_n/\PP^1$, where $n=K^2_\eta$ and $\eta$ is the generic fibre. 

It is known that a $dP_n/\PP^1$ is not birationally rigid when the total space is smooth and $n\geq 4$. For $n>4$ the $3$-fold is rational (see for example \cite{isk}), hence non-rigid, and it was shown in \cite{alexeev} that $dP_4/\PP^1$ are birational to conic bundles. 

Understanding conditions under which a $dP_n/\PP^1$, for $n\leq 3$, is birationally rigid  is a key step in providing the full picture of MMP, and hence the classification, in dimension three. Birational rigidity for the smooth models of degree $1,2$ and $3$ is well studied, see for example \cite{Pu98b}. However, as we see later, while the smoothness assumption for degree $3$ is only a generality assumption, considering the smooth case for $n=1,2$ is not very natural. Hence the necessity of considering singular cases is apparent.

In this article, I focus on a well-known conjecture (Conjecture~\ref{conjecture-Grinenko}) on this topic that connects birational rigidity of del Pezzo fibrations of low degree to the structure of their mobile cone. A counterexample to this conjecture is provided when the 3-fold admits certain singularities. 

{\bf Acknowledgement.} 
I would like to thank Gavin Brown for his support during my PhD, where I learnt most of the techniques used in this article. I would also like to thank Ivan Cheltsov and Mikhail Grinenko for showing interest in this example. I am also grateful to the anonymous referees, whose valuable comments helped me shape the presentation of this paper.

\section{Grinenko's conjecture}

Pukhlikov in \cite{Pu98b} proved that a general smooth $dP_3/\PP^1$ is birationally rigid if the class of $1$-cycles $mK^2_X-L$ is not effective for any $m\in\Z$, where $L$ is the class of a line in a fibre. This condition is famously known as the $K^2$-condition.

\begin{defi}\label{k2} A del Pezzo fibration is said to satisfy $K^2$-condition if the $1$-cycle $K^2$ does not lie in the interior of the Mori cone $\NNE$.\end{defi}

The birational rigidity of smooth $dP_n/\PP^1$ for $d=1,2$ was also considered in \cite{Pu98b} and the criteria for rigidity are similar to that for $n=3$. 

In a sequential work \cite{Grinenko, Gri2, Gri, Gri1, Gri3, Gri4, Gri5}, Grinenko realised and argued evidently that it is more natural to consider $K$-condition instead of the $K^2$-condition.

\begin{defi}A del Pezzo fibration is said to satisfy $K$-condition if the anticanonical divisor does not lie in the interior of the Mobile cone.\end{defi}

\begin{rmk} It is a fun, and not difficult, exercise to check that $K^2$-condition implies $K$-condition. And the implication does not hold in the opposite direction.\end{rmk}

One of the most significant observations of Grinenko was the following theorem.

\begin{theorem}\cite{Grinenko, Gri} Let $X$ be a smooth $3$-fold Mfs, with del Pezzo surfaces of degree $1$ or $2$, or a general degree $3$, fibres over $\PP^1$. Then $X$ is birationally rigid if it satisfies the $K$-condition.\end{theorem}

He then conjectured that this must hold in general, as formulated in the conjecture below, with no restriction on the singularities.

\begin{conj}\label{conjecture-Grinenko}(\cite{Grinenko} Conjecture~1.5 and \cite{Gri2} Conjecture~1.6) Let $X$ be a $3$-fold Mori fibration of del Pezzo surfaces of degree $1,2$ or $3$ over $\PP^1$. Then $X$ is birationally rigid if and only if it satisfies the $K$-condition.\end{conj}

It is generally believed that Grinenko's conjecture might hold if one only considers Gorenstein singularities. 

Note that, it is not natural to only consider the smooth case for $n=1,2$ as these varieties very often carry some orbifold singularities inherited from the ambient space, the non-Gorenstein points. For example a del Pezzo surface of degree $2$ is naturally embedded as a quartic hypersurface in the weighted projective space $\PP(1,1,1,2)$. It is natural that a family of these surfaces meets the singular point $1/2(1,1,1)$. See \cite{Ahm} for construction of models and the study of their birational structure. 

Grinenko also constructed many nontrivial (Gorenstein) examples, which supported his arguments. The study of quasi-smooth models of $dP_2\slash\PP^1$ in \cite{Ahm}, i.e. models that typically carry a quotient singularity, also gives evidence that the relation between birational rigidity and the position of $-K$ in the mobile cone is not affected by the presence of the non-Gorenstein point. Below in \S\ref{unstabledP2example} I give a counterexample to Conjecture~\ref{conjecture-Grinenko} for a Gorenstein singular degree $2$ del Pezzo fibration. 

On the other hand, in \cite{BCZ}, Example 4.4.4, it was shown that this conjecture does not hold in general for the degree $3$ case (of course in the singular case) and  suggested that one must consider the semi-stability condition on the $3$-fold $X$ in order to state an updated conjecture:

\begin{conj}[\cite{BCZ}, Conjecture~2.7] Let $X$ be a $dP_3\slash\PP^1$ which is semistable in the sense of Koll\'ar\cite{Kollar-stability}. Then $X$ is birationally rigid if it satisfies the $K$-condition.\end{conj}

Although this type of (counter)examples are very difficult to produce, the expectation is that such example in degree $1$ is possible to be produced. On the other hand, a notion of (semi)stability for del Pezzo fibrations of degree $1$ and $2$ seems necessary (as already noted in \cite{Co00}  Problem~5.9.1), in order to state an improved version of Conjecture~\ref{conjecture-Grinenko}, and yet there has been no serious attempt in this direction.

\section{The counterexample}

The most natural construction of smooth $dP_2/\PP^1$ is the following, due to Grinenko \cite{Grinenko}.

Let $\mathcal{E}=\mathcal{O}\oplus\mathcal{O}(a)\oplus\mathcal{O}(b)$ be a rank $3$ vector bundle over $\PP^1$ for some positive integers $a,b$, and let $V=\Proj_{\PP^1}\mathcal{E}$. Denote the class of the tautological bundle on $V$ by $M$ and the class of a fibre by $L$ so that 
\[\Pic(V)=\Z[M]+\Z[L]\]
Assume $\sigma\colon X\rightarrow V$ is a double cover branched over a smooth divisor $R\sim 4M-2eL$, for some integer $e$. The natural projection $p\colon V\rightarrow\PP^1$ induces a morphism $\pi\colon X\rightarrow\PP^1$, such that the fibres are del Pezzo surfaces of degree $2$ embedded as quartic surfaces in $\PP(1,1,1,2)$. This $3$-fold $X$ can also be viewed as a hypersurface of a rank two toric variety. Let $T$ be a toric fourfold with Cox ring $\C[u,v,x,y,z,t]$, that is $\Z^2$-graded by
\begin{equation}\label{smooth-model}\left(\begin{array}{cccccc}
u&v&x&y&z&t\\
1&1&0&-a&-b&-e\\
0&0&1&1&1&2
\end{array}\right)\end{equation}
The $3$-fold $X$ is defined by the vanishing of a general polynomial of degree $(-2e,4)$. It is studied in \cite{Grinenko} for which values $a,b$ and $e$ this construction provides a Mfs, and then he studies their birational properties.
This construction can also be generalised to non-Gorenstein models \cite{Ahm}. As before, let $X$ be defined by the vanishing of a polynomial of degree $(-n,4)$ but change the grading on $T$ to
\begin{equation}\label{sing-model}\left(\begin{array}{cccccc}
u&v&x&y&z&t\\
1&1&-a&-b&-c&-d\\
0&0&1&1&1&2
\end{array}\right)\end{equation}
where $c$ and $d$ are positive integers. When $X$ is a Mfs, it is easy to check that $X$ is Gorenstein if and only if $e=2d$. See \cite{Ahm} for more details and construction. 

The cone of effective divisors modulo numerical equivalence on $T$ is generated by the toric principal divisors, associate to columns of the matrix above, and we have $\Eff(T)\subset\Q^2$. This cone decomposes, as a chamber, into a finite union of subcones 
\[\Eff(T)=\bigcup \Nef(T_i)\]
where $T_i$ are obtained by the Variation of Geometric Invariant Theory (VGIT) on the Cox ring of $T$. See \cite{cox} for an introduction to the GIT construction of toric varieties, and \cite{BZ} for a specific treatment of rank two models and connections to Sarkisov program via VGIT. In \cite{BZ} it was also shown how the toric 2-ray game on $T$ (over a point) can, in principle, be realised from the VGIT. I refer to \cite{Co00} for an explanation of the general theory of 2-ray game and Sarkisov program. In certain cases, when the del Pezzo fibration is a Mori dream space, i.e. it has a finitely generated Cox ring, its 2-ray game is realised by restricting the 2-ray game of the toric ambient space. In other words, in order to trace the Sarkisov link one runs the 2-ray game on $T$, restricts it to $X$ and checks whether the game remains in the Sarkisov category, in which case a winning game is obtained and a birational map to another Mfs is constructed. See \cite{Ahm, AZ, BZ, dP4, AK, BCZ} for explicit constructions of these models for del Pezzo fibrations or blow ups of Fano $3$-folds. I demonstrate this method in the following example, which also shows that Conjecture~\ref{conjecture-Grinenko} does not hold.

\subsection{Construction of the example}\label{unstabledP2example} Suppose $T$ is a toric variety with
the Cox ring $\Cox(T)=\C[u,v,x,t,y,z]$, grading given by the matrix
\[A=\left(\begin{array}{cccccc}
u&v&x&t&y&z\\
1&1&0&-2&-2&-4\\0&0&1&2&1&1\end{array}\right),\]
and let the irrelevant ideal $I=(u,v)\cap(x,y,z,t)$. In other words $T$ is the geometric quotient with character $\psi=(-1,-1)$.  Suppose $\mathcal{L}$ is the linear system of divisors of degree $(-4,4)$ in $T$; I use the notation $\mathcal{L}=|\mathcal{O}_T(-4,4)|$. This linear system is generated by monomials according to the following table
\[\begin{array}{c||c|c|c|c|c|c|c}
\deg \text{ of }u,v\text{ coefficient}&0&2&4&6&8&10&12\\
\hline
&&&&&&&\\
\text{fibre monomials}&x^3z&xy^3&yzt&xyz^2&xz^3&yz^3&z^4\\ 
&t^2&ty^2&xy^2z&tz^2&y^2z^2&&\\
&x^2y^2&x^2yz&y^4&y^3z&&&\\
&xyt&xzt&x^2z^2&&&&
\end{array}\]     

Let $g$ be a polynomial whose zero set defines a general divisor in $\mathcal{L}$ (I use the notation $g\in\mathcal{L}$). Then, for example, $x^3z$ is a monomial in $g$ with nonzero coefficient, and appearance of $y^4$ in the column indicated by $4$ means that $a(u,v)y^4$ is a part of $g$, for a homogeneous polynomial $a$ of degree $4$ in the variables $u$ and $v$, so that  $a(u,v)y^4$ has bidegree $(-4,4)$.
           
Now consider a sublinear system $\mathcal{L}'\subset\mathcal{L}$ with the property that $u^i$ divides the coefficient polynomials according to the table
\[\begin{array}{c||c|c|c|c|c|c|c|c|c}
\text{power of }u&1&2&3&4&5&6&8&9&12\\
\hline
&&&&&&&&&\\
\text{monomial}&x^2yz&xzt&yzt&x^2z^2&xyz^2&tz^2&xz^3&yz^3&z^4\\
&&xy^2z&y^3z&&&y^2z^2&&&
\end{array}\]
and general coefficients otherwise. And suppose $f\in\mathcal{L}'$ is general. For example $u^{12}z^4$ is a monomial in $f$ and no other monomial that includes $z^4$ can appear in $f$. Or, for instance, $xyz^2$ in the column indicated by $5$  carries a coefficient $u^5$; in other words, we can only have monomials of the form $\alpha u^6xyz^2$ or $\beta vu^5xyz^2$ in $f$, for $\alpha ,\beta\in\mathbb{C}$, and no other monomial with $xyz^2$ can appear in $f$. Denote by $X$ the hypersurface in $T$ defined by the zero locus of $f$.\\

\paragraph{\bf Argumentation overview.} Note that the $3$-fold $X$ has a fibration over $\PP^1$ with degree $2$ del Pezzo surfaces as fibres.  In the remainder of this section, I check that it is a Mori fibre space and then I show that the natural 2-ray game on $X$ goes out of the Mori catergory.  This is by explicit construction of the game. It is verified, from the construction of the game, that $-K_X\notin\Int\MMob(X)$. Hence the conditions of Conjecture~\ref{conjecture-Grinenko} are satisfied for $X$. In Section~\ref{new-model}, a new (square) birational model to $X$ is constructed, for which the anticanonical divisor is interior in the mobile cone. Then I show that it admits a birational map to a Fano $3$-fold, and hence it is not birationally rigid.

\begin{lm}\label{sing}  The hypersurface $X\subset T$ is singular. In particular $\sing(X)=\{p\}$, where $p=(0:1:0:0:0:1)$. Moreover, the germ at this point is of type $cE_6$, and $X$ is terminal.
\end{lm}

\begin{proof} By Bertini theorem $\sing(X)\subset\Bs(\mathcal{L}')$. Appearance of $y^4$ in $\mathcal{L}'$ with general coefficients of degree $4$ in $u$ and $v$ implies that this base locus is contained in the loci $(y=0)$ or $(u=v=0)$. However, $(u=v=0)$ is not permitted as $(u,v)$ is a component of the irrelevant ideal. Hence we take $(y=0)$. Also $t^2\in\mathcal{L}'$ implies $\Bs(\mathcal{L}'\subset (y=t=0)$.The remaining monomials are $x^3z$, $u^4x^2z^2,u^8xz^3$ and $u^{12}z^4$, where $b$ is a general quartic. In fact, appearance of the first monomial, i.e. $x^3z$, implies that $x=0$ or $z=0$. And $u^{12}z^4\in\mathcal{L}'$ implies $z=0$ or $u=0$. Therefore
\[\Bs(\mathcal{L}')=(u=x=y=t=0)\cup(y=t=z=0).\]

The part $(y=t=z=0)$ is the line $(u:v;1:0:0:0)$, which is smooth because of the appearance of the monomial $x^3z$ in $f$. And the rest is exactly the point $p$.

Note that near the point $p\in T$, I can set $v\neq 0$ and $z\neq 0$, to realise the local isomorphism to $\mathbb{C}^4$, which is

\[\Spec\mathbb{C}[u,v,x,y,z,t,\frac{1}{v},\frac{1}{z}]^{\mathbb{C}^*\times\mathbb{C}^*}=\Spec\mathbb{C}[\frac{u}{v},\frac{x}{v^4z},\frac{t}{v^6z^2},\frac{y}{v^2z}]\]
I denote the new local coordinates by $\u,\x,\y,\t$. Now looking at $f$, in this local chart, we observe that the only quadratic part is $\t^2$, and the cubic part is $\x^3$. The variable $\y$ appears in degree $4$, after some completing squares with $\t$ and $\x$ and $\u$ has higher order. In particular, after some analytic changes we have
\[f_\text{loc}\equiv\t^2+\x^3+\y^4+\u.(\text{higher order terms})\]
which is a $cE_6$ singularity.

\end{proof}

In the following lemma, I study the 2-ray game played on $X\slash\{\text{pt}\}$. This will be used to figure out the shape of $\Mob(X)$ and the position of $-K_X$ against it.

\begin{lm}\label{2-ray} The 2-ray game of $T$ restricts to a game on $X$.\end{lm}
\begin{proof}The 2-ray game on $T$ goes as follows
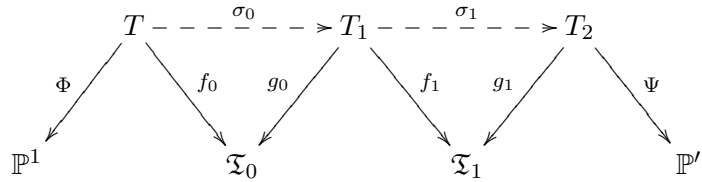
\[\xymatrixcolsep{2pc}\xymatrixrowsep{3pc}
\xymatrix{
&T\ar_\Phi[ld]\ar^{f_0}[rd]\ar^{\sigma_0}@{-->}[rr]&&T_1\ar_{g_0}[ld]\ar^{f_1}[rd]\ar^{\sigma_1}@{-->}[rr]&&T_2\ar_{g_1}[ld]\ar^\Psi[rd]&\\
\PP^1&&\mathfrak{T}_0&&\mathfrak{T}_1&&\PP^\prime
}\]
The varieties in this diagram are obtained as follows. The GIT chamber of $T$ is indicated by the matrix $A$ and it is 
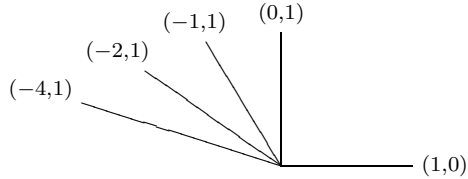
\[\xygraph{
!{(0,0) }="a"
!{(1.7,0) }*+{\scriptstyle{(1,0)}}="b"
!{(0,1.6) }*+{\scriptstyle{(0,1)}}="c"
!{(-0.9,1.5) }*+{\scriptstyle{(-1,1)}}="d"
!{(-1.7,1.2) }*+{\scriptstyle{(-2,1)}}="e"
!{(-2.5,0.8) }*+{\scriptstyle{(-4,1)}}="f"
"a"-"b"  "a"-"c" "a"-"d" "a"-"e" "a"-"f"
} \]
The ample cone of the variety $T$ is the interior of the cone \[\convex\left<(1,0),(0,1)\right>\]
and $T_1$ corresponds to the interior of the cone \[\convex\left<(0,1),(-1,1)\right>,\]
i.e., this cone is the nef cone of $T_1$. In other words $T_1$ is a toric variety with the same coordinate ring and grading as $T$
but with irrelevant ideal $I_1=(u,v,x)\cap(t,y,z)$. Similarly $T_2$ corresponds to the cone generated by rays $(-1,1)$ and $(-2,1)$, i.e., the irrelevant ideal of $T_2$ is $I_2=(u,v,x,t)\cap(y,z)$. On the other hand, the four varieties in the second row of the diagram correspond to the $1$-dimensional rays in the chamber. The projective line $\PP^1$ corresponds to the ray generated by $(1,0)$; monomial of degree $(n,0)$ form the graded ring $\C[u,v]$. Similarly $\mathfrak{T}_0$ correspond to the ray generated by $(0,1)$. In other words,
$\mathfrak{T}_0$ is
\[\Proj\bigoplus_{n\geq 1}H^0(T,\mathcal{O}_T(0,n))=\Proj\C[x, u^2t, uvt,v^2t, u^2y,\dots,v^4z],\]
which is embedded in $\PP(1^9,2^3)$ via the relations among the monomials above.
The varieties $\mathfrak{T}_1$ and $\PP^\prime$ can also be explicitly computed in this way. In particular, $\PP^\prime=\PP(1,1,2,4,6)$.

For the maps in the diagram we have that
\begin{enumerate}
\item $\Phi\colon T\rightarrow \PP^1$ is the natural fibration.
\item $f_0\colon T\rightarrow\mathfrak{T}_0$ is given in coordinate by
\[(u,v,x,t,y,z)\in T\longmapsto(x,u^2,uvt,v^2t,u^2y,\dots,v^4z)\in\mathfrak{T}_0\subset\PP(1^9,2^3)\]
It is rather easy to check that away from $p=(1:0:\dots:0)\in\mathfrak{T}_0$ the map $f_0$ is one-to-one. The pre-image of this point under $f_0$ is the set $(u=v=0)\cup(t=y=z=0)$. But $(u=v=0)$ corresponds to a component of the irrelevant ideal, and hence it is empty on $T$. This implies that the line $(t=y=z=0)\cong\PP^1\subset T$ is contracted to a point via $f_0$. In particular, $\mathfrak{T}_0$ is not $\Q$-factorial. Similarly $g_0\colon T_1\rightarrow\mathfrak{T}_0$ is the contraction of the surface $(u=v=0)$, to the same point in $\mathfrak{T}_0$. In particular, $\sigma_0$ is an isomorphism in codimension~$1$. 

Let us have a look at the local description of these maps. In $\mathfrak{T}_0$ consider the open set given by $x\neq 0$, a neighbourhood of $p$. This affine space is (with an abuse of notation for local coordinates)
\[\Spec\C[u^2t, uvt,v^2t, u^2y,\dots,v^4z].\]
And at the level of $T$ (respectively $T_1$) using $\{x\neq0\}$ I can get rid of the second grading (corresponding to the second row of the matrix $A$) and obtain a quasi-projective variety which is the quotient of $\C^5-\{u=v=0\}$ (respectively $\C^5-\{t=y=z=0\}$) by an action of $\C^*$ by 
\[(\lambda;(u,v,t,y,z)\mapsto(\lambda u,\lambda v,\lambda^{-2}t,\lambda^{-2}y,\lambda^-4z)\]
I denote this by $(1,1,-2,-2,-4)$ anti-flip, following \cite{gavin, miles-flip}. See these references for an explanation of this construction and notation. This all means that under $\sigma_0$ a copy of $\PP^1$ is replaces by a copy of $\PP(2,2,4)$. 
\item Similarly, $\sigma_1$ is an isomorphism in codimension $1$. Note that the action of $(\C^*)^2$ by $A$ is invariant under $\Gl(2,\Z)$ action on $A$, so that multiplying it from the left by 
\[\left(\begin{array}{cc}
1&1\\
0&1
\end{array}\right)\]
and setting $t=1$, as in the previous case, shows that this map is of type $(1,1,1,-1,-3)$. So $f_1$ contracts a copy of $\PP^2$ and $g_1$ extracts a copy of $\PP(1,3)$. 
\item As noted before, the irrelevant ideal of $T_2$ is $I_2=(u,v,x,t)\cap(y,z)$, and $\Psi\colon T_2\rightarrow\PP^\prime=\PP(1,1,2,4,6)$ is the contraction of the divisor $E=(z=0)$ to a point in $\PP^\prime$. This map, similar to $\Phi$, is given by 
\[\bigoplus_{n\geq 1}H^0(T,\mathcal{O}_T(0,n))\]
where the degrees are now considered in the transformation of $A$ by the action of
\[\left(\begin{array}{cc}
1&2\\
0&1
\end{array}\right)\]
from the left.
\end{enumerate}

As already noted $\Phi$ restricted to $X$ is just the degree 2 del Pezzo fibration. The restriction of $f_0$ contracts the same line, to the same point on a subvariety of $\mathfrak{T}_0$. Once we set $x=1$, a linear form ``$z$'' appears in $f$, hence locally in the neighbourhood where the flip is happening one can eliminate this variable. Therefore, $\sigma_0$ restricts to $(1,1,-2,-2)$, which means a copy of $\PP^1$ is contracted to a point and a copy of $\PP(2,2)\cong\PP^1$ is extracted. The restriction of $\sigma_1$ to the $3$-fold $X_1$, i.e., the image of $X$ under the anti-flip, is an isomorphism. This is because $t^2$ is a term in $f$, and hence the toric flip happens away from the $3$-fold. Finally $\Psi$ restricts to a divisorial contraction to $X_1$.
\end{proof}

\begin{rmk} Note that it follows from some delicate version of Lefschetz principle that $\rank\Pic(X)=2$. In fact, $X$ is defined by a linear system of bi-degree $(-4,4)$, and does not belong to the nef cone of $T$, which is generated by $(1,0)$ and $(0,1)$. However, it is in the interior of the mobile cone of $T$. In particular, it is nef and big on both $T_1$ and $T_2$, which are isomorphic to $T$ in codimension $1$. Hence we have that
\[\Pic(X)\cong\Pic(T_1)\cong\Pic(T)\cong\Z^2\]
These isomorphisms follow a singular version of Lefschetz hyperplane theorem as in \cite[\S2.2]{stratified}, and in this particular case it holds because $(-4,4)$ represents an interior point in the cone of mobile divisors on $T$.
A more detailed argument for this can be found in \cite[\S4.3]{Ahm} or \cite[Proof of Proposition 32]{BCZ}. Also note that $X$ and $X_1$ are isomorphic, so it does not really matter which one to consider for these arguments.
\end{rmk}

\begin{prop}\label{bad-link} The 2-ray game on $X$ obtained in Lemma~\ref{2-ray} does not provide a new Mfs model of $X$.
\end{prop}
\begin{proof} This is quite clear now. The first reason for the failure of the game is the anti-flip $(1,1,-2,-2)$. As mentioned before, in this anti-flip a copy of $\PP^1$ is contracted to a point and on the other side of the anti-flip a copy of $\PP(2,2)\cong\PP^1$ is extracted. In particular, the map replaces a smooth line by a line that carries singularities at each point of it. Note that the line itself is smooth (isomorphic to the projective line) but on the $3$-fold it is singular (at each point). As terminal singularities are isolated, this game goes out of the Mori category. Another reason for the failure is that $-K_X\in\partial\MMob(X)$, as explained below. This means that in the last map of the 2-ray game of $X$ the contracted curves are trivial against the canonical divisor, hence not fulfilling the rules of Mori theory.\end{proof}

\begin{rmk} It is also a good point here to observe that the anti-canonical class of $X$ has degree $(-2,1)$. This can be seen as follows. The variety $T$ is toric, hence its anticanonical divisor is given by the sum of the toric principal divisors, in particular it has degree $(-6,5)$, note that I am still working with the matrix $A$ and this bi-degree is nothing but the sum of the columns of $A$.
By adjunction formula $K_X=(K_T+X)_{|_X}$, which implies that $-K_X\sim\mathcal{O}(-2,1)$. On the other hand, the fact that the 2-ray game on $X$ is essentially the restriction of the 2-ray game on $T$ implies that $\MMob(X)$ and $\MMob(T)$, as convex cones in $\Q^2$, have the same boundaries. In fact, the decomposition of $\MMob(X)$ into the union of nef cones is a sub-decomposition of $\MMob(T)$. In particular
\[\MMob(X)=\convex\left<(1,0),(-2,1)\right>,\]
which implies 
\[-K_X\in\partial\MMob(X).\]
And this shows that $X$ satisfies conditions of Conjecture~\ref{conjecture-Grinenko}.
\end{rmk}

\section{The Fano variety birational to $X$}\label{new-model}

Now consider the fibrewise transform $T\rightarrow\F$, given by
\[(u,v,x,t,y,z)\mapsto(u,v,u^4x,u^6t,u^3y,z)\]
where $\F$ is a toric variety with Cox ring $\Cox(\F)=\C[u,v,x,t,z,y]$, with irrelevant ideal $I_\F=(u,v)\cap(x,y,z,t)$, and the grading
\[A^\prime=\left(\begin{array}{cccccc}
u&v&x&t&z&y\\
1&1&0&0&0&-1\\0&0&1&2&1&1\end{array}\right)\]
Denote by $X^\prime$ the birational transform of $X$ under this map.
And suppose $g$ is the defining equation of $X^\prime$. The polynomial $g$ can be realised from $f$ by substituting $x,y,z,t$ by $u^4x,u^3y,z,u^6t$, and then cancelling out $u^{12}$ from it. In particular it is of the form
\[g=t^2+x^3z+z^4+x^2y^2+a_4(u,v)y^4+\{\text{ other terms}\}\]
In fact the full table of monomials appearing in $g$ is given by

\[\begin{array}{c||c|c|c|c|c}
\deg \text{ of }u,v\text{ coefficient}&0&1&2&3&4\\
\hline
&&&&&\\
\text{fibre monomials}
&x^3z&x^2yz&xy^2z&y^3z&y^4\\
&x^2z^2&xyz^2&y^2z^2&&\\
&xz^3&yz^3&y^2t&&\\
&z^4&yzt&uxy^3&&\\
&t^2&&&&\\
&xzt&&&&\\
&z^2t&&&&\\
&u^2x^2y^2&&&&\\
&uxyt&&&&\\
\end{array}\]     

Note that this table does not generate a general member in the linear system $|\mathcal{O}_\mathbb{F}(0,4)|$, and the missing monomials are $x^4,tx^2$ and $x^3y$. Also $x^2y^2$, $xyt$ and $xy^3$ do not have general coefficients in $u,v$, as specified in the table above.

\begin{theorem} The $3$-fold $X^\prime$ is smooth and it is birational to a Fano $3$-fold. In particular, $X^\prime$, and hence $X$, are not birationally rigid.
\end{theorem}
\begin{proof} 
A similar, and easier, check to that of Lemma~\ref{sing} shows that $X^\prime$ is smooth: First note that the base locus of the linear system is the line $(u:v;1:0:0:0)$, given by $y=z=t=0$. Then observe that any point on this line is smooth, guaranteed by appearance of the monomial $x^3z\in g$.

Now, let us play the 2-ray game on $\mathbb{F}$ and restrict it to $X'$.
The 2-ray game on $\F$ proceeds, after the fibration to $\PP^1$, by a divisorial contraction to $\PP(1,1,1,1,2)$. This is realised by
\[\bigoplus_{n\geq 1}H^0(\F,\mathcal{O}_\F(0,n))\]
and is given in coordinates by
\[(u:v;x:t:z:y)\mapsto(x:z:uy:vy:t)\]
In particular, the divisor $E:(y=0)\subset\F$ is contracted to the locus $\PP(1,1,2)\subset\PP(1,1,1,1,2)$. The restriction of this map to $X^\prime$ shows that the divisor $E_X=E\cap X$ is contracted to a quartic curve in $\PP(1,1,2)$. The image of $X^\prime$ under this map is a quartic $3$-fold in $\PP(1,1,1,1,2)$, which is a Fano variety of index $2$. Note that, similar to $X$, we can check that $-K_{X^\prime}\sim\mathcal{O}_{X^\prime}(1,1)$. In particular, $-K_{X^\prime}$ is ample, and of course interior in the mobile cone. This variety is also studied in \cite{Ahm} Theorem~3.3.

\end{proof}


\def\cprime{$'$}
\providecommand{\bysame}{\leavevmode\hbox to3em{\hrulefill}\thinspace}
\providecommand{\MR}{\relax\ifhmode\unskip\space\fi MR }
\providecommand{\MRhref}[2]{%
  \href{http://www.ams.org/mathscinet-getitem?mr=#1}{#2}
}
\providecommand{\href}[2]{#2}

\vspace{1.5cm}
\noindent Radon Institute, Austrian Academy of Sciences,\\
 Altenberger Str. 69, A-4040 Linz, Austria\\
e-mail: \url{hamid.ahmadinezhad@oeaw.ac.at}


\end{document}